\documentclass[10pt,reqno]{amsart}
\usepackage{geometry}
\usepackage{amssymb}
\usepackage[breaklinks]
{hyperref}

\allowdisplaybreaks

\newcommand{\C}{\mathbb{C} }
\newcommand{\R}{\mathbb{R} }
\newcommand{\Z}{\mathbb{Z} }

\newcommand{\N}{\mathbb{N} }
\newcommand{\Lpnorm}[2]{\left\| #2 \right\|_{L^{#1}}}

\newcommand{\convolvedwith}{*}

\newcommand{\indicator}[1]{{\bf 1}_{#1}}

\newcommand{\sizeof}[1]{\left \lvert #1 \right \rvert}
\newcommand{\absolutevalueof}[1]{\left \lvert #1 \right \rvert}

\newcommand{\setof}[1]{\left\{ #1 \right\}}

\newcommand{\dimension}{d}

\renewcommand{\time}{m}
\newcommand{\largetime}{M}

\newcommand{\maximalfunction}{M}
\newcommand{\variation}{V_r}
\newcommand{\jump}{\lambda}
\newcommand{\jumpfunction}{N}

\DeclareMathOperator{\supp}{supp}

\newtheorem{prop}{Proposition}[section]
\newtheorem{lemma}{Lemma}[section]

\newtheorem{theorem}{Theorem}

\newtheorem*{Moon}{Moon's theorem}
\newtheorem*{CdG}{Carrillo--de Guzm\'an's theorem}
\theoremstyle{definition}

\theoremstyle{remark}

\title[Pointillist principle for variations]{The Pointillist Principle for Variation Operators and Jump functions}
\author[K Hughes]{Kevin Hughes}
\address{
School of Computing, Engineering \& the Built Environment, 
Edinburgh Napier University, 
Edinburgh EH10 5DT, UK
}
\email{khughes.math@gmail.com}


\begin{document}

\maketitle
%\tableofcontents
% 
% NOTES to myself
% 
%command to box things in amsmath \boxed{...}
%

\begin{abstract}
I extend the pointillist principles of Moon and Carrillo--de Guzm\'an to variational operators and jump functions. 
\end{abstract}

% ---
\section{The pointillist principle}
% ---

In \cite{Moon}, Moon observed that, for a sequence of sufficiently smooth convolution operators and any $q \geq 1$, the weak \((1,q)\) boundedness of their maximal operator is equivalent to restricted weak \((1,q)\) boundedness of the maximal operator. In this paper, the goal is to extend this theorem to variational operators and to jump functions. 
I now recall a couple definitions in order to make this precise. 

For a sequence of operators $(T_\time)_{\time \in \N}$, define their maximal function 
\[
\maximalfunction(T_\time f(x) : \time \in \N) := \sup_{\time \in \N} |T_\time f(x)|
\]
for $f:\R^\dimension \to \C$ and $x \in \R^\dimension$. 
%$T_* f(x) := \sup_{\time \in \N} |T_\time f(x)|$. 
Suppose that \(p,q \geq 1\). An operator \(T\) is \emph{weak-type \((p,q)\) with norm \(C\)} if it satisfies the inequality
\begin{equation}\label{def:weak-type}
\| Tf \|_{L^{q,\infty}} \leq C \|f\|_{L^p}
\quad \text{for all} \quad f \in L^p
\end{equation}
where \( \|f\|_{L^p} := \big( \int |f(x)|^p dx \big)^{1/p} \) and \( \| g \|_{L^{q,\infty}} := \sup_{t>0} t|\{x \in \R^\dimension : |g(x)| \geq t\}|^{1/q} \) for functions \(f,g : \R^d \to \C\) with the usual modifications made when \(p\) or \(q\) is infinite. 
% \[
% \|f\|_{L^p} 
% := 
% \big( \int |f(x)|^p dx \big)^{1/p}
% \quad \text{and} \quad 
% \| g \|_{L^{q,\infty}} 
% := 
% \sup_{t>0} t|\{|g| \geq t\}|^{1/q}
% .\]
% Combining these, we find that the former is equivalent to 
% \begin{align*}
% \sup_{t>0} t|\{|g| \geq t\}|^{1/q} \leq C \|f\|_{L^p}
% \end{align*}
% and 
% \[
% |\{|g| \geq t\}| 
% \leq \big( Ct^{-1} \|f\|_{L^p} \big)^{q} 
% \quad \text{for all } t>0
% .\]
% \[
% \sup_{t>0} \; t\Big| \big\{x \in \R^\dimension : |g(x)| \geq t \big\} \Big|^{1/q} 
% \leq C \big( \int |f(x)|^p dx \big)^{1/p}
% .\]
Here and throughout, \(C\) is non-negative. 
In this paper, we will restrict our functions to be defined on \(\R^\dimension\) and will work with the Lebesgue measure thereon. So, I will rarely include this in the notation, and I will also let \(|X|\) denote the measure of a finite (Lebesgue) measurable set \(X\) in \(\R^\dimension\). 
Additionally, an operator \(T\) is said to be \emph{restricted weak-type \((p,q)\) with norm \(C\)} if \eqref{def:weak-type} holds for each function \(f\) which is the characteristic function of a finite measurable set. 

\begin{Moon}%[\cite{Moon}]
Suppose that $(T_\time)_{\time \in \N}$ is a sequence of convolution operators given by $T_\time f := f \convolvedwith g_\time$ with $g_\time \in L^1(\R^\dimension)$ for each \(\time\in\N\). 
For any $q \geq 1$, $\maximalfunction(T_\time f(x) : \time \in \N)$ is restricted weak-type $(1,q)$ with norm $C$ if and only if $\maximalfunction(T_\time f(x) : \time \in \N)$ is weak-type $(1,q)$ with norm $C$. 
\end{Moon}

The essential difference between the two distinct weak-types lies in the class of input functions used to define them. The class of all \(L^p\) functions serve as input to the (unqualified) weak-type inequalities while its subclass of characteristic functions of finite measurable sets serve as input to the restricted weak-type inequalities. In particular, if an operator is weak-type \((p,q)\) then it is automatically restricted weak-type \((p,q)\). 
Moon's theorem says that the converse is true for certain maximal functions when \(p=1\). The converse may fail for linear operators when \(p>1\); see page~149 in \cite{Moon}. 
% Interestingly, Akcoglu--Baxter--Bellow--Jones showed that the analogue of Moon's theorem over the discrete case $\Z$ (as opposed to \(\R^\dimension\)) may fail; see \cite{ABBJ:fail} and \cite{HJ:fail}. 

In \cite{CdG}, Carillo--de Guzm\'an gave a version of Moon's theorem where the class of characteristic functions is replaced by linear combinations of delta functions. 
To state their result, we introduce more terminology. 
Let $\delta_x$ denote the (Dirac) delta function at the point $x \in \R^\dimension$. 
In analogy with restricted weak-type, let us say that an operator \(T\) is \emph{pointed weak-type $(p,p)$ with norm at most $C$} if for any finite subset of points $X \subset \R^\dimension$, we have the inequality 
\begin{equation}\label{eq:restricted_pointed_weak_type_inequality}
\Big| \big\{x \in \R^\dimension : \Big|{T \big(\sum_{y \in X} \delta_y \big)(x)} \Big| > \lambda \big\} \Big|
\leq 
% C \cdot \frac{\# X}{\lambda^{p}} 
C\#{X} / \lambda^{p} 
\quad \text{for all } \lambda>0
.\end{equation}
This inequality and definition is to be interpreted as defined only when the operator \(T\) makes sense on delta functions. For instance, this makes sense when \(T\) is taken to be the maximal function \( \maximalfunction(f*g_\time(x) : \time \in \N) \) formed from a sequence of \(L^1\) functions \((g_\time)_{\time \in \N}\) in which case \( \sum_{y \in X} \delta_y * g_\time(x) = \sum_{y \in X} g(x-y) \). 
\begin{CdG}%[\cite{CdG}]
Suppose that \((T_\time)_{\time \in \N}\) is a sequence of convolution operators given by $T_\time f := f \convolvedwith g_\time$ with $g_\time \in L^1(\R^\dimension)$ for each \(\time\in\N\). 
For any $p \geq 1$, $\maximalfunction(T_\time f(x) : \time \in \N)$ is weak-type $(p,p)$ with norm at most C if $\maximalfunction(T_\time f(x) : \time \in \N)$ is pointed weak-type $(p,p)$ with norm at most C. 
Furthermore, the converse is true if $p=1$. 
\end{CdG}

% \begin{remark}
Pointed weak-type inequalities form a third distinct class of inequalities because finite sums of delta functions serve as input to the pointed weak-type inequalities; note that delta functions are not \(L^p\) functions for any \(p\) and give a distinct class of input functions. 
The converse to Carrillo--de Guzm\'an's theorem can fail for $p>1$; see page~121 in \cite{CdG}.
% \end{remark}

Grafakos--Mastylo extended Moon's theorem to the multilinear setting in \cite{GM} while Carena extended Carrillo--de Guzm\'an's theorem to more general metric measure spaces in \cite{Carena}. 
See \cite{K} and \cite{MS} for more extensions. 
It is this collection of theorems we refer to as the `pointillist principle', taking its name from the Pointillism movement in art. 

The purpose of this short note is to extend Moon and Carrillo--de Guzm\'an's instances of the pointillist principle to variational operators and jump functions. 
The pointillist principle led to a new proof of boundedness of the Hardy--Littlewood maximal function in \cite{Carlsson} and the best constant for the Hardy--Littlewood maximal function in one dimension in \cite{Melas}, and it is my hope that this work will be used to give new proofs of the \(L^p\) boundedness of the variation of Hardy--Littlewood averages. 
I now recall these operators and discuss a few of their basic properties. 

Let $r \in [1,\infty)$ and $\mathcal{R} \subseteq \N$. Suppose that $(f_\time)_{\time \in \N}$ is a sequence of Lebesgue measurable functions. Define pointwise the $r$-variation of the subsequence $(f_\time)_{\time \in \mathcal{R}}$ 
\begin{equation}
\label{def:variation}
\variation(f_\time(x) : \time \in \mathcal{R})
:= 
% \sup \inparentheses{\sum_{i=1}^L |f_{\time_i}(x) - f_{\time_{i+1}}(x)|^r}^{1/r} 
\sup \Big( {\sum_{i=1}^L |f_{\time_i}(x) - f_{\time_{i+1}}(x)|^r} \Big)^{1/r} 
,\end{equation}
where the supremum is over all finite, increasing subsequences $(\time_i)$ in $\mathcal{R}$. 
One makes the usual modification using the essential supremum to extend \eqref{def:variation} to $r=\infty$. 
%For \( \lambda > 0\) the jump function $\mathcal{N}_\jump(T_\time f(x) : \time \in \mathcal{R})$ is the supremum over all $M \in \N$ (includes 0) such that there exists \( t_0 < t_1 < \cdots < t_M \) in $\mathcal{R}$ with 
%\begin{equation*}
%|T_{t_i}f(x) - T_{t_{i+1}}f(x)| > \lambda
%\end{equation*}
%for each $0 \leq i \leq M-1$. 
Note that $\variation(\cdot)$ is sublinear in its argument. 
For \(\lambda>0\), define the jump function $\jumpfunction_\lambda(f_\time(x) : \time \in \mathcal{R})$ as given by the supremum over $M \in \N$ such that 
there exists a sequence \( s_0 < t_0 \leq s_1 < t_1 \leq \cdots  \leq s_M < t_M \) in $\mathcal{R}$ with 
% \begin{equation*}
\( |f_{s_i}(x) - f_{t_i}(x)| > \lambda \) 
% \end{equation*}
for all $0 \leq i \leq M$. 
%These are related by the readily verified inequality: 
%\begin{equation*}
%\jumpfunction_\jump(T_\time f(x) : \time \in \mathcal{R})
%\leq 
%\mathcal{N}_{\lambda}(T_\time f(x) : \time \in \mathcal{R})
%\leq 
%\jumpfunction_{\jump/2}(T_\time f(x) : \time \in \mathcal{R})
%.
%\end{equation*}
Unlike the variation operators, the jump functions fail to be sublinear. 
However, we note the almost sub-additivity of the jump functions: 
\begin{equation}\label{equation:jumpfunction_sublinearity}
\jumpfunction_\jump([f_\time+g_\time](x) : \time \in \mathcal{R}) 
\leq 
\jumpfunction_{\lambda_1}(f_\time(x) : \time \in \mathcal{R}) + \jumpfunction_{\lambda_2}(g_\time(x) : \time \in \mathcal{R})  
\end{equation}
for \(\lambda_1\) and \(\lambda_2\) positive with \(\lambda_1+\lambda_2=\lambda\).

For present purposes, we are most interested in these objects when the functions \(f_\time := T_\time{f}\) for a sequence of operators \((T_\time)_{\time\in\N}\) e.g., naturally occurring families of linear operators in probability and analysis such as expectation operators from a martingale or Hardy--Littlewood averages. The main problem becomes establishing the \(L^p\) boundedness of the associated variation operators and jump functions. 

The variation operators are connected to the jump functions by the inequality: 
\begin{equation*}
\jumpfunction_\lambda(T_\time f(x) : \time \in \mathcal{R}) 
\leq 
4 \lambda^{-r} [\variation(T_\time f(x) : \time \in \mathcal{R})]^r
\end{equation*}
for each $r \geq 1$. 
Surprisingly this can be reversed on average in $L^p(\R^\dimension)$ for $1 \leq p < \infty$ when $r>2$. 
In practice the $L^p$ boundedness of $V_2$ often fails. 
However the jump function \( \lambda \sqrt{N_\lambda} \) may still be bounded in which case it acts as a surrogate `endpoint' operator for $V_2$; see \cite{JSW}. 
The variation operators are related to the maximal functions by 
\begin{align*}
V_\infty(T_\time f(x) : \time \in \mathcal{R})
&= 
2\maximalfunction(T_\time f(x) : \time \in \mathcal{R})
\\&\leq 
2\left[ V_\infty(T_\time f(x) : \time \in \mathcal{R}) + T_{\time_0} f(x) \right]
\end{align*}
for any $\time_0 \in \mathcal{R}$. 
Because of this inequality, we may henceforth assume that \(r\) is finite. 
On the one hand, $\variation f(x)$ increases as $r$ decreases so that its $L^p$-boundedness becomes more difficult to prove. 
On the other hand, the jump inequalities and variational estimates give quantitative versions of pointwise ergodic theorems. 
For a more thorough discussion of variations and jump functions, see \cite{Bourgain_pointwise, PX, JSW, MST}. 

Our first theorem generalizes Moon's theorem to variations and jump functions. 
\begin{theorem}%[Moon's theorem for variations and jump functions]
\label{theorem:Moon_variations}
Suppose that $(T_\time)_{\time \in \N}$ is a sequence of convolution operators given by $T_\time f := f \convolvedwith g_\time$ with $g_\time \in L^1(\R^\dimension)$ for each \(\time\in\N\). 
% Fix the variation exponent to be any $r \in [1,\infty]$. 
For any $q,r \geq 1$, $\variation(T_\time f : \time \in \N)$ is restricted weak-type $(1,q)$ with norm $C$ if and only if $\variation(T_\time f : \time \in \N)$ is weak-type $(1,q)$ with norm $C$. 
Moreover, $\lambda \sqrt[r]{\jumpfunction_\jump}$is restricted weak-type $(1,q)$ if and only if $\lambda \sqrt[r]{\jumpfunction_\jump}$is weak-type $(1,q)$. 
\end{theorem}

We also prove the Carrillo--de Guzm\'an version of Theorem~\ref{theorem:Moon_variations}.
\begin{theorem}%[Pointed weak-type for variation and jump functions.]
\label{theorem:CdG_variations}
Suppose that $(T_\time)_{\time \in \N}$ is a sequence of convolution operators given by $T_\time f := f \convolvedwith g_\time$ with $g_\time \in L^1(\R^\dimension)$ for each \(\time\in\N\). 
% Fix the variation exponent to be any $r \in [1,\infty]$. 
If $p,r \geq 1$ and $\variation(T_\time f : \time \in \N)$ is pointed weak-type $(p,p)$ with norm $C$, then $\variation(T_\time f : \time \in \N)$ is strong-type $(p,p)$ with norm at most $C$. 
Moreover the same is true for the jump functions $\lambda \sqrt[r]{\jumpfunction_\jump}$. 
\end{theorem}

We can extend Theorem~\ref{theorem:Moon_variations} to a slightly more general set-up. 
In addition to working with convolutions of $L^1$ functions, we will work with convolutions of smoothing, possibly singular, measures. 
This extension appeared for the maximal function of lacunary dilates of a smoothing measure in unpublished work of Seeger--Tao--Wright connected with \cite{STW_pointwise_lacunary}. %\cite{STW_Radon}
Inspired by the set-up of \cite{SW:lacunary_problems}, we use a weak version of condition (2) of Seeger--Wright's Theorem~1.1 in \cite{SW:lacunary_problems}. 
Let $(\mu_\time)_{\time \in \N}$ be a sequence of finite measures of bounded variation and $T_\time$ denote convolution with $\mu_\time$. 
Assume that, for some fixed $p \geq 1$ and for each \( \largetime \in \N \), we have
\begin{equation}\label{eq:smoothing_measures}
\sup_{\time \leq \largetime}  \|T_\time \circ P_{>k}\|_{L^p \to L^p} 
= 
o(1) 
\text{ as } k \to \infty
.\end{equation}
Here, and throughout, $P_k$ denotes a smooth Littlewood--Paley `projection' operator adapted to frequency band of frequency size $2^k$. 
To be precise, let $\indicator{[-1,1]} \leq \phi \leq \indicator{[-2,2]}$ be a smooth function on $\R$.  Define by the multiplier $\widehat{P_k}(\xi) = \phi(|\xi|) - \phi(2|\xi|)$. 
Then for a function $f : \R^\dimension \to \C$, $\widehat{P_k f} := \widehat{P_k} \cdot \widehat{f}$ the Fourier transform of $P_k f$ has support in $\{ |\xi| \in [2^{k-1},2^{k+1}] \}$ while $\sum_{k \in \Z} \phi(|\xi|) - \phi(2|\xi|) \equiv 1$ for $\xi \in \R^\dimension$ so that $\sum_{k \in \Z} P_k f = f$ in many senses. 
We write $P_{\leq k} f = \sum_{j \leq k} P_j f$ and $P_{> k} f = \sum_{j > k} P_j f$. 
As a motivating example one may consider the lacunary spherical averages given by the measures $\mu_\time := \sigma_{2^\time}$ for $\time \in \N$ where $\sigma_{r}$ is the spherical measure on a sphere of radius $r>0$ normalized to have mass 1. 
It is known that $\|P_k \mu_r\|_{L^2(\R^\dimension)} \lesssim (1+r2^{-k})^{\frac{1-\dimension}{2}}$ for $\dimension \geq 2$ so that \eqref{eq:smoothing_measures} is satisfied for these examples. 
%
%Consider another Littlewood--Paley operator $Q_k$ defined as before but replacing the multiplier $\phi(|\xi|)$ with $\phi(|\xi|/2)$; that is, $\hat{Q_k}(\xi) := \phi(|\xi|/2) - \phi(|\xi|)$. 
%Note that $P_k \circ Q_k = Q_k \circ P_k = P_k$. 
%
%For instance if $\mu_n$ is the spherical measure on the $\dimension-1$ dimensional sphere of radius $2^{-n}$ normalized to be a probability measure, then 
%\[
%f*\mu_n(x) 
%= 
%\int_{S^{\dimension-1}} f(x-2^{-n}y) \; d\sigma(y)
%\]
%makes sense for continuous functions. 

We have the following `smoothing' version of Moon's theorem and Theorem~\ref{theorem:Moon_variations}. 
\begin{theorem}
%[Moon's theorem for variations of smoothing averages]
\label{theorem:Moon_smoothing}
Suppose that $(T_\time)_{\time \in \N}$ is a sequence of convolution operators given by $T_\time f := f \convolvedwith \mu_\time$ where \(\mu_\time\) is a finite measure of bounded total variation satisfying the smoothing property \eqref{eq:smoothing_measures} for each \(m\in\N\). 
% Fix the variation exponent to be any $r \in [1,\infty]$. 
For any $q,r \geq 1$, $\variation(T_\time f : \time \in \N)$ is restricted weak-type $(1,q)$ with norm $C$ if and only if $\variation(T_\time f : \time \in \N)$ is weak-type $(1,q)$ with norm at most $C$. 
Moreover, $\lambda \sqrt[r]{\jumpfunction_\jump}$is restricted weak-type $(1,q)$ if and only if $\lambda \sqrt[r]{\jumpfunction_\jump}$is weak-type $(1,q)$. 
\end{theorem}

%\begin{remark}
%The motivating example for our smoothing measure is the spherical measure. 
%Suppose that $\mu_r$ is the induced Lebesgue measure on the sphere of radius $r$ centered at the origin normalized to unit mass. 
%Then $\maximalfunction(f * \mu_{2^\time} : \time \in \Z)$ is the lacunary spherical maximal function of Calder\'on and Coifman--Weiss and is known to be bounded on $L^p(\R^\dimension)$ for all $p>1$. 
%$\|P_k \mu_r\|_{L^2(\R^\dimension)} \lesssim (1+r2^{-k})^{\frac{1-\dimension}{2}}$ so that \eqref{??} is easily seen to be satisfied while 
%\[
%\|P_{\leq k} \mu_r\|_{L^1(\R^\dimension)} 
%\leq 
%\|\phi(\cdot/2^k)\check{}\;\|_{1} \cdot \|\mu_r\|_{TV} 
%\lesssim 
%2^{k\dimension} \cdot 2^{-k(\dimension-1)} = 2^k 
%.\]
%\end{remark}

%We end our introduction by noting that the pointillist principle recently appeared in dynamics under the guise of `pixelation'. 
%See the blog post \href{http://lamington.wordpress.com/2014/10/31/dipoles-and-pixie-dust/}{Dixie poles and Pixie dust} on \href{http://lamington.wordpress.com/}{Geometry and the Imagination} for a similar idea to the pointillist principle and for references. 

We close the introduction with a little bit of notation that will be useful in the proof of our theorems. 
First, $f(x) \lesssim g(x)$ if there exists a constant $f(x) \leq C g(x)$ for some implicit constant $C>0$. 
Second, for a subset $F \subset \R^\dimension$, let $\indicator{F}$ denote the indicator or characteristic function of $F$.

% ---
\section*{Acknowledgements}
% ---

I thank Jim Wright for informing me of the extension of Moon's theorem to smoothing mesaures, as well as for insightful conversations. 
I thank Ben Krause, Mariusz Mirek and Bartosz Trojan for discussions concerning variational operators and jump function inequalities while at the \href{https://www.mathematics.uni-bonn.de/him/programs/past/tp_2014_05}{2014 Trimester Program: ``Harmonic Analysis and Partial Differential Equations"} at the Hausdorff Research Institute for Mathematics. 
I thank the referee for their feedback which improved the exposition of this paper.

%\bigskip
% --
\section{Moon's theorem for variations}
% --

The proof of Moon's theorem hinges on how to approximate simple functions. 
The following proposition is implicit in \cite{Moon}. 
It says that the set $I_\epsilon$ approximates $f$ very well, in the sense that it has the same size as $f$ and it is close to the convolution of $f$ with a prescribed finite sequence of smooth functions. 
Since we will use it in the proof of Theorem~\ref{theorem:Moon_variations}, we include its proof for completeness. 
\begin{prop}[Moon's pointillist principle]
\label{proposition:Moon_approximation}
For a finite sequence $(h_\time)_{\time\in[\largetime]}$ of $C^1(\R^\dimension)$ functions, if $f$ is a simple function on $\R^\dimension$, then for any $\epsilon>0$, there exists a set $I_\epsilon \subseteq \supp(f)$ such that 
\begin{enumerate}
\item 
$\Lpnorm{\infty}{f} \sizeof{I_\epsilon} = \Lpnorm{1}{f}$, 
\item 
$\absolutevalueof{f \convolvedwith h_\time(x) - (\Lpnorm{\infty}{f} \indicator{I_\epsilon}) \convolvedwith h_\time(x)} < \epsilon \Lpnorm{1}{f}$ for $\time \in [\largetime]$ and all $x \in \R^\dimension$. 
\end{enumerate}
\end{prop}

\begin{proof}[Proof of Proposition~\ref{proposition:Moon_approximation}]
By the scaling homogeneity of the problem we may normalize $\Lpnorm{\infty}{f}=1$. 
Let $f = \sum_{k=1}^K a_k \indicator{F_k}$ be a simple function with coefficients $a_k \in \R$ and each set $F_k \subset \R^\dimension$ of finite Lebesgue measure. 
We may assume that the $F_k$ are open balls with diameter at most $\delta>0$ a small parameter that we will optimize later. 
Let $I_k$ be \emph{any} open ball in $F_k$ such that $\sizeof{I_k} = a_k \sizeof{F_k}$. 
Now set $I = \cup_{k} I_k$ so that $\Lpnorm{\infty}{f} |I| = |I|$. 
%\begin{equation}
%\Lpnorm{\infty}{f} \cdot \sizeof{I} = \Lpnorm{1}{f} 
%. 
%\end{equation}

We want to show that the difference between $f$ and $\indicator{I} = \Lpnorm{\infty}{f} \indicator{I}$ is small. 
First note that 
\begin{align*}
f \convolvedwith h_\time(x) 
& = 
\int_{\R^\dimension} \sum_k a_k \indicator{F_k}(y) h_\time(x-y) \; dy 
%\\ & = 
= \sum_k a_k \int_{F_k} h_\time(x-y) \; dy 
\\ & = 
\sum_k a_k \sizeof{F_k} h_\time(x-y_k) 
%\\ & = 
= \sum_k \Lpnorm{\infty}{f} \sizeof{I_k} h_\time(x-y_k) 
\\ & = 
\sum_k \sizeof{I_k} h_\time(x-y_k) 
\end{align*}
for some $y_k \in F_k$ since the $h_\time$ are smooth by the Mean Value Theorem. 
Similarly since $I_k \subset F_k$, we can write 
\begin{align*}
\indicator{I} \convolvedwith h_\time(x) 
& = 
\int_{\R^\dimension} \sum_k \indicator{I_k}(y) h_\time(x-y) \; dy 
\\ 
& = 
\sum_k \int_{I_k} h_\time(x-y) \; dy 
= 
\sum_k \sizeof{I_k} h_\time(x-y_k') 
\end{align*}
for some $y_k' \in I_k$. 
Therefore we have the pointwise estimate 
\begin{align*}
\absolutevalueof{f \convolvedwith h_\time(x) - \indicator{I} \convolvedwith h_\time(x)} 
& = 
\absolutevalueof{\sum_k \sizeof{I_k} h_\time(x-y_k) - \sum_k \sizeof{I_k} h_\time(x-y_k') \; dy} 
\\ & \leq 
%\sum_k \absolutevalueof{\sizeof{I_k} h_\time(x-y_k) - \sizeof{I_k} h_\time(x-y_k')} 
%\\ & = 
\sum_k \sizeof{I_k} \cdot \absolutevalueof{h_\time(x-y_k) - h_\time(x-y_k')} 
.\end{align*}
Since the functions $h_\time$ are smooth and $\largetime$ is finite, we can choose $\delta$ small enough so that 
\(
\absolutevalueof{h_\time(x-y_k) - h_\time(x-y_k')} < \epsilon
\) 
for each $1 \leq \time \leq \largetime$. 
Take $I_\epsilon$ to be $I$ to conclude the proof.
\end{proof}

We will make use of the following inequality multiple times. 
\begin{lemma}
If \(1 \leq p,r \leq \infty\) and \((f_m)_{m\in[M]}\) is a finite sequence of \(L^p\)-functions, then 
\begin{equation}\label{bound:union}
\| \variation(f_m : \time \in [\largetime]) \|_{L^p}
\leq 
2M^2 \sup_{\time \in [\largetime]} \|f_\time\|_{L^p}
.\end{equation}
\end{lemma}

\begin{proof}
Fix \(1 \leq p,r \leq \infty\). 
First note the pointwise inequality 
\begin{equation*}%\label{bound:union}
\variation(f_m(x) : \time \in [\largetime]) 
% \leq 
% V_1(f*g_\time(x) : \time \in [\largetime]) 
\leq 
2M \sup_{\time \in [\largetime]} |f_\time(x)| 
.\end{equation*}
This inequality follows from using the fact that \(V_r\) increases as \(r\) decreases and then applying the triangle inequality to \(V_1\). 
Next take \(L^p\) norms, replace the supremum by a sum, and use the triangle inequality to find that 
\begin{align*}
\| \variation(f_m : \time \in [\largetime]) \|_{L^p}
\leq 
2M \| \sup_{\time \in [\largetime]} |f_\time| \|_{L^p}
\leq 
2M^2 \sup_{\time \in [\largetime]} \|f_\time\|_{L^p}
.\end{align*}
This is the desired inequality. 
\end{proof}

With \eqref{bound:union} and Proposition~\ref{proposition:Moon_approximation} in hand, let us prove Theorem~\ref{theorem:Moon_variations}. 

\begin{proof}[Proof of Theorem~\ref{theorem:Moon_variations}]
Weak-type obviously implies restricted weak-type so we only prove that restricted weak-type implies weak-type. 
Fix $q,r \geq 1$. 
% We make two standard reductions. 
By Monotone Convergence, reduce to the truncated variation operator 
$\variation(f*g_\time(x) : \time \in [\largetime])$ 
where the supremum is over all finite, increasing subsequences of $[\largetime] := \setof{1, \dots, \largetime}$ as long as our bounds at the end are independent of $\largetime$. 
Normally one would also reduce to simple functions, however we cannot do this since we do not yet know that the variation operator is continuous. 
Assume for now that $f$ is a simple function, and we will remove this restriction at the end of the argument. 
By dilational symmetry of $L^1(\R^\dimension)$, normalize our simple function so that $\Lpnorm{\infty}{f}=1$. 
Let \(\lambda>0\). 

Let $\epsilon>0$. 
Our first step is to approximate $g_\time \in L^1(\R^\dimension)$ by smooth $h_\time \in L^1(\R^\dimension)$. 
%(By smooth we only need that each $h_\time$ is continuously differentiable.)
We can do this so that $\Lpnorm{1}{g_\time-h_\time} < \epsilon$ for each \(\time \in [\largetime]\). 
Then, for each $x \in \R^\dimension$ 
\begin{equation*}%\label{eq:measure_approximation}
\absolutevalueof{f \convolvedwith (g_\time-h_\time)(x)} 
\leq 
\Lpnorm{\infty}{f} \Lpnorm{1}{g_\time-h_\time}
< 
\epsilon 
.\end{equation*}
% Sublinearity of the variation operators implies 
% \begin{equation}\label{equation:variation_sublinearity}
% \begin{split}
% \variation(f*g_\time(x) : \time \in [\largetime]) 
% & \leq 
% \variation(f*[g_\time-h_\time](x) : \time \in [\largetime])  
% \\&\qquad + \variation(f*h_\time(x) : \time \in [\largetime]) 
% .\end{split}
% \end{equation}
Applying this and the inequality \eqref{bound:union} with \(f_\time := f*[g_\time-h_\time]\) and \(p=\infty\) implies that 
\begin{equation*}
\| \variation(f*[g_\time-h_\time] : \time \in [\largetime]) \|_{L^\infty}
< 
2\largetime^2 \epsilon 
% \quad \text{for all } x \in \R^\dimension
.\end{equation*}
% The union bound then implies the pointwise bound
% \begin{align*}
% \variation(f*g_\time(x) : \time \in [\largetime]) 
% %& \leq 
% %\variation(f: \{g_\time-h_\time\}_{\time \in [\largetime]})(x) + \variation(f: \{h_\time\}_{\time \in [\largetime]})(x)  
% %\\ 
% & < 
% \largetime^{1/r} \cdot \epsilon + \variation(f*h_\time(x) : \time \in [\largetime]) 
% .\end{align*}

Apply Proposition~\ref{proposition:Moon_approximation} to find a subset $I_\epsilon \subset \supp(f)$ such that $\sizeof{I_\epsilon} = \Lpnorm{1}{f}$ and satisfying the inequality $\absolutevalueof{f \convolvedwith h_\time(x) - \indicator{I_\epsilon} \convolvedwith h_\time(x)} < \epsilon$ simultaneously for each $\time \in [\largetime]$ and every $x \in \R^\dimension$. 
This latter condition implies that for any $\time_1, \time_2 \in [\largetime]$ and $x \in \R^\dimension$, 
\[
\absolutevalueof{(f-\indicator{I_\epsilon}) \convolvedwith h_{\time_1}(x) - (f-\indicator{I_\epsilon}) \convolvedwith h_{\time_2}(x)} 
%\leq 
%\absolutevalueof{(f-I_\epsilon) \convolvedwith g_\time(x)} + \absolutevalueof{(f-I_\epsilon) \convolvedwith g_m(x)} 
< 
2\epsilon 
.\]
Applying this and \eqref{bound:union} with \(f_\time := [f-\indicator{I_\epsilon}]*h_\time\) and \(p=\infty\) implies that 
\[
\Lpnorm{\infty}{\variation([f-\indicator{I_\epsilon}]*h_\time : \time \in [\largetime]) } 
< 
4\largetime^2 \epsilon 
.\]

% Since $\variation$ is sublinear, we have 
% \begin{align*}
% \variation(f*h_\time(x) : \time \in [\largetime])
% &\leq 
% \variation([f-\indicator{I_\epsilon}]*h_\time(x) : \time \in [\largetime]) 
% \\& \qquad	+ \variation(\indicator{I_\epsilon}*h_\time(x) : \time \in [\largetime]) 
% \\&\leq 2\epsilon\largetime^{1/r} + \variation(\indicator{I_\epsilon}*h_\time(x) : \time \in [\largetime]) 
% .\end{align*}
Let \(\delta \in (0,1)\) and choose $\epsilon = \delta/(8\largetime^2)$. The above inequalities imply that 
\begin{align*}
\sizeof{\setof{\variation(f*g_\time(x) : \time \in [\largetime]) > \lambda+2\delta}} 
& \leq 
\sizeof{\setof{\variation(f*[g_\time-h_\time](x) : \time \in [\largetime]) > \delta}} 
\\ & \quad + \sizeof{\setof{\variation([f-\indicator{I_\epsilon}]*h_\time(x) : \time \in [\largetime]) > \delta}} 
\\ & \quad + \sizeof{\setof{\variation(\indicator{I_\epsilon}*h_\time(x) : \time \in [\largetime]) > \lambda}} 
\\ &= 
\sizeof{\setof{\variation(\indicator{I_\epsilon}*h_\time(x) : \time \in [\largetime]) > \lambda}} 
%\\ & 
%\leq 
%C (\lambda/3)^{-q} \sizeof{I_\epsilon} 
%\\ & 
%= 
%3^q C \lambda^{-q} \Lpnorm{1}{f}
.\end{align*} 
%where the inequality is an application of our hypothesis that the restricted weak-type bound holds. 
Applying our hypothesis that the variation is restricted weak-type $(1,q)$, we find 
\begin{equation*}
\sizeof{\setof{\variation(f*g_\time(x) : \time \in [\largetime]) > \lambda+2\delta}} 
\leq 
C \lambda^{-q} \sizeof{I_\epsilon} 
= 
C \lambda^{-q} \Lpnorm{1}{f}
.\end{equation*}
Taking \(\delta\) to 0, we obtain the desired bound for simple functions. 

We extend our estimates to $f$ in $L^1(\R^\dimension)$. 
Find a simple function \( s := \sum_{k=1}^K a_k \indicator{F_k} \) where the subsets $F_k \subset \R^\dimension$ have finite Lebesgue measure and \( \|f - s\|_{L^1(\R^\dimension)} < \delta \) where \(\delta \in (0,1) \) is a parameter which we will optimize in a moment. 
% \begin{equation*}
% \variation(T_\time(f-s)(x) : \time \in [\largetime]) 
% % \leq V_1(T_\time(f-s)(x) : \time \in [\largetime]) 
% \leq 2 \sum_{\time \in [\largetime]} |T_\time(f-s)(x)| 
% \end{equation*}
% Then 
% \begin{align*}
% \big\{ |\variation(T_\time(f-s)(x) : \time \in [\largetime])| > \epsilon \big\}
% \subseteq 
% \bigcup_{\time \in [\largetime]} \big\{ |T_\time(f-s)(x)| > \epsilon/(2M) \big\}
% .\end{align*}
% Taking the measure of these sets and using Chebyshev's inequality followed by Young's inequality, we find that 
% \begin{align*}
% \Big| \big\{ |\variation(T_\time(f-s)(x) : \time \in [\largetime])| > \epsilon \big\} \Big|
% &\leq 
% \sum_{\time \in [\largetime]} \Big| \big\{ |T_\time(f-s)(x)| > \epsilon/(2M) \big\} \Big| 
% \\&\leq 
% \largetime \sup_{\time \in [\largetime]} \Big| \big\{ |T_\time(f-s)(x)| > \epsilon/(2M) \big\} \Big| 
% \\&\leq 
% 2\largetime^2\epsilon^{-1} \sup_{\time \in [\largetime]} \|T_\time(f-s)\|_{L^1}
% \\&\leq 
% 2\largetime^2\epsilon^{-1}\delta \sup_{\time \in [\largetime]} \|g_\time\|_{L^1}
% .\end{align*}
The bound \eqref{bound:union} implies
\begin{align*}
\|\variation(T_\time(f-s) : \time \in [\largetime])\|_{L^1} 
&\leq 
2\largetime^2 \sup_{\time \in [\largetime]} \| T_\time(f-s) \|_{L^1} 
.\end{align*}
Young's inequality implies that 
\(
\| T_\time(f-s) \|_{L^1} 
<
\delta \|g_\time\|_{L^1}
\) for all \(\time\). 
Therefore, 
\begin{align*}
\|\variation(T_\time(f-s) : \time \in [\largetime])\|_{L^1} 
< 
2\largetime^2 \delta \sup_{\time \in [\largetime]} \|g_\time\|_{L^1}
.\end{align*}
Chebyshev's inequality implies that for each positive \(\epsilon\) we have the bound 
\[
\Big| \big\{ |\variation(T_\time(f-s)(x) : \time \in [\largetime])| > \epsilon \big\} \Big|
< 
2\largetime^2\epsilon^{-1}\delta \sup_{\time \in [\largetime]} \|g_\time\|_{L^1}
\]
Choosing \(\delta = \epsilon^2/(2\largetime^2 \sup_{\time \in [\largetime]} \|g_\time\|_{L^1})\), we see that there exists a simple function \(s\) such that \( | \{ |\variation(T_\time(f-s)(x) : \time \in [\largetime])| > \epsilon \} | < \epsilon \). 
Using the sublinearity of the variation operators, we finally obtain 
\begin{align*}
| \{ |\variation(T_\time{f}(x) : \time \in [\largetime])| > \lambda+\epsilon \} | 
&\leq 
| \{ |\variation(T_\time(f-s)(x) : \time \in [\largetime])| > \epsilon \} | 
\\&\qquad+ | \{ |\variation(T_\time{s}(x) : \time \in [\largetime])| > \lambda \} | 
\\&< 
\epsilon + | \{ |\variation(T_\time{s}(x) : \time \in [\largetime])| > \lambda \} | 
.\end{align*}
Using the bound previously established for simple functions and taking \(\epsilon\) to 0 completes the proof. 

The proof for jump inequalities is similar but replaces the use of sublinearity for variation operators with almost sub-additivity of jump functions \eqref{equation:jumpfunction_sublinearity}.
%\begin{equation}
%\label{equation:jumpfunction_sublinearity}
%\mathcal{N}_\lambda(f*g_\time(x) : \time \in [\largetime]) 
%\leq 
%\mathcal{N}_{\lambda/2}(f*[g_\time-h_\time](x) : \time \in [\largetime])  + \mathcal{N}_{\lambda/2}(f*h_\time(x) : \time \in [\largetime])
%.\end{equation}
% Unfortunately we lose a small power of 2 in this inequality and thus cannot conclude that the norm bounds are the same. 
Breaking up \(\lambda\) into \(\lambda_1+\lambda_2\) and taking one of the parameters to 0 allows us to obtain the same constants. 
\end{proof}

%\begin{remark}
%Note that we use sublinearity of the variation operators in two different aspects: one for my input function $f$, and another for the families of operators given by convolution with $g_\time$ and $g_\time-h_\time$. 
%\end{remark}

Our strategy for the proof of Theorem~\ref{theorem:Moon_smoothing} is to take $h_\time$ to be $P_{\leq k} \mu_\time$ for some large \(k\) as an approximation to $\mu_\time$ and bound the rest as error.  
We assumed that $\mu_\time$ is a finite measure of bounded total variation so that $P_k \mu_\time$, which is the convolution of $\mu_\time$ with a Schwartz function, is well-defined, and $\| P_{\leq k} \mu_\time \|_{p} \lesssim_k \|\mu_\time\|_{TV}$ where $\|\mu_\time\|_{TV}$ denotes the total variation of $\mu_\time$. 
We remark that the implicit bound is not uniform in $k$; this presents a minor technicality.

\begin{proof}[Proof of Theorem~\ref{theorem:Moon_smoothing}]
Once again, weak-type immediately implies restricted weak-type; so, we only prove the converse. 
Fix the exponents \(q,r \geq 1\). Assume that \(q < \infty\); the modifications for \(q=\infty\) are left to the reader. 
Reduce to the truncated variation operator $\variation(f*\mu_\time : \time \in [\largetime])$ for large $\largetime \in \N$ as before. %Once again, our bounds will be independent of \(M\). 
For the moment choose $f$ to be a simple function normalized so that $\|f\|_\infty = 1$. 
%We know $\|f\|_q$ is finite for $1 \leq q \leq \infty$. 
Let \(\lambda>0\). 

Let \(\epsilon, \delta \in (0,1)\). 
Choose $k$ sufficiently large so that assumption \eqref{eq:smoothing_measures} implies that
\[
\| f * P_{>k}\mu_\time \|_{L^p} 
< 
\epsilon \| f \|_{L^p}
\]
uniformly in $\time \in [\largetime]$. 
%\emph{At the moment $k$ depends on $\epsilon$ which in turn depends on $\lambda$, $f$ and $\largetime$. This is ok as long as the implicit constants are independent of $\lambda$ which they currently are.} 
The bound \eqref{bound:union} yields 
\begin{align*}
\| \variation(f*P_{>k}\mu_\time : \time \in [\largetime]) \|_{L^p}
\leq 
2\largetime^{2} \sup_{\time \leq \largetime} \| f \convolvedwith P_{>k} \mu_\time \|_{L^p}
< 
2\largetime^{2} \epsilon \|f\|_{L^p}
.\end{align*}
Chebyshev's inequality implies that 
\begin{align*}
|\{ \variation(f*P_{>k}\mu_{\time} : \time \in [\largetime]) > \delta \}| 
&\leq 
\delta^{-p} \| \variation(f*P_{>k}\mu_\time : \time \in [\largetime]) \|_{L^p}^p
\\&<
\delta^{-p} (2\largetime^{2} \epsilon \|f\|_{L^p})^p
.\end{align*}

Apply Proposition~\ref{proposition:Moon_approximation} with $g_\time := P_{\leq k}\mu_\time$ to find a subset $I_\epsilon$ satisfying the conclusions of Proposition~\ref{proposition:Moon_approximation}. 
Replacing \(f\) by \(\indicator{I_\epsilon}\) in the above analysis shows that
\begin{align*}
|\{ \variation(\indicator{I_\epsilon}*P_{>k}\mu_{\time} : \time \in [\largetime]) > \delta \}| 
<
\delta^{-p} (2\largetime^{2} \epsilon \|\indicator{I_\epsilon}\|_{L^p})^p
.\end{align*}
From our assumption on $I_\epsilon$ in the conclusion of Proposition~\ref{proposition:Moon_approximation} with \(h_\time := P_{\leq k}\mu_\time\), \eqref{bound:union} also implies that 
\begin{align*}
\| \variation([f-\indicator{I_\epsilon}]*P_{\leq k}\mu_\time : \time \in [\largetime]) \|_{L^\infty}
&\leq 
2\largetime^{2} \epsilon \sup_{\time \in [\largetime]} \| [f-\indicator{I_\epsilon}]*P_{\leq k}\mu_\time \|_{L^\infty}
\\&<  
2\largetime^{2} \epsilon \| f \|_{L^1}
.\end{align*}

The decomposition, 
\begin{align*}
f*\mu_{\time} 
= 
f*P_{>k}\mu_{\time} 
+ [f-\indicator{I_\epsilon}]*P_{\leq k}\mu_{\time} 
+ \indicator{I_\epsilon}*\mu_{\time} - \indicator{I_\epsilon}*P_{>k}\mu_{\time} 
,\end{align*}
implies 
\begin{align*}
|\{ \variation(f*\mu_{\time} : \time \in [\largetime]) > \lambda + 3\delta \}| 
& \leq 
|\{ \variation(f*P_{>k}\mu_{\time} : \time \in [\largetime]) > \delta \}| 
\\ & \qquad + |\{ \variation([f-\indicator{I_\epsilon}]*P_{\leq k}\mu_{\time} : \time \in [\largetime]) > \delta \}| 
\\ & \qquad + |\{ \variation(\indicator{I_\epsilon}*P_{>k}\mu_{\time}: \time \in [\largetime]) > \delta \}| 
\\ & \qquad + |\{ \variation(\indicator{I_\epsilon}*\mu_{\time}: \time \in [\largetime]) > \lambda \}| 
.\end{align*}
Let \(X = \max\{ 1, \|f\|_{L^p}, \|\indicator{I_\epsilon}\|_{L^p}, \|f\|_{L^1} \}\)
and choose \( \epsilon = \delta^2/(8M^{2}X) \) to obtain 
\begin{align*}
|\{ \variation(f*\mu_{\time} : \time \in [\largetime]) > \lambda + 3\delta \}| 
& \leq 
2\delta^p + |\{ \variation(\indicator{I_\epsilon}*\mu_{\time}: \time \in [\largetime]) > \lambda \}| 
.\end{align*}
Applying the restricted weak-type hypothesis and letting \(\delta\) tend to 0 completes the proof for simple functions. 

To extend from simple functions to all $f$ in $L^1$, adapt the approximation argument at the end of the proof of Theorem~\ref{theorem:Moon_variations}. 
Finally, the adaptation to jump functions is analogous to before. 
\end{proof}

%\bigskip
%\newpage
% --
\section{Carrillo--de Guzm\'an's Theorem for variations}
% --

The proof of Theorem~\ref{theorem:CdG_variations} will be similar to that of Carrillo--de Guzm\'an's theorem and Theorem~\ref{theorem:Moon_variations} using the following proposition as the Carrillo--de Guzm\'an analogue of Proposition~\ref{proposition:Moon_approximation}. 
\begin{prop}\label{proposition:CdG_approximation}
Let $(g_\time)_{\time \in [\largetime]}$ be a finite sequence  of uniformly continuous functions, and $f = \sum_{k=1}^K a_k \indicator{F_k}$ be a simple function on $\R^\dimension$ with $F_k$ dyadic cubes from the standard dyadic mesh on $\R^\dimension$. 
If $\epsilon>0$, then $f$ can be refined into a sum of dyadic cubes $f = \sum b_j \indicator{Q_j}$ where $Q_j$ is in some $F_k$, and for any points $y_j$ in the interior of $Q_j$, we have 
\begin{equation}\label{CdG:pointwise_comparison}
\big| {f \convolvedwith g_\time(x) - \sum_j b_j |Q_j| g_\time(x-y_j)} \big| 
< 
\Lpnorm{1}{f} \epsilon
\end{equation}
for each $1 \leq \time \leq \largetime$ and all $x \in \R^\dimension$. 
\end{prop}

% This is similar to Proposition~\ref{proposition:Moon_approximation}, but has the technical advantage of choosing $F_k$ dyadic to partition $\R^\dimension$. 
%This more closely relates to the $p$-adic spaces in the next section. 
% Alternatively we could use open balls with diameter $<\delta$ that partition $\R^\dimension$. 
% The main point of Proposition~\ref{proposition:CdG_approximation} is the pointwise comparison \eqref{CdG:pointwise_comparison}. 

\begin{proof}[Proof of Proposition~\ref{proposition:CdG_approximation}]
Since each of the $g_\time$ are uniformly continuous and there are finitely many of them, they are altogether uniformly continuous. 
This means that for any $\epsilon>0$, which we pick and fix now, if $|x-y|<\delta$, then $|g_\time(x)-g_\time(y)| < \epsilon$ simultaneously for all $\time$. 
With this in mind, use the dyadic structure in $\R^\dimension$ to decompose each dyadic cube $F_k$ into a finite union $\cup_{\ell} Q_{k,\ell}$ of dyadic cubes whose interiors are disjoint and each of which has diameter at most $\delta$. 
% Then 
% \[
% \Lpnorm{1}{f - \sum_k a_k \indicator{\cup_{j} Q_{k,j}} } 
% < 
% (\sum_{k=1}^K |a_k|) \epsilon 
% = 
% \Lpnorm{1}{f} \epsilon 
% .\]
Partitioning and reordering the cubes and coefficients as necessary, we rewrite $f = \sum_{j} b_j \indicator{Q_j}$. 
Let $y_j$ be a point in the interior of $Q_j$. 
For each cube $Q_j$ and $x \in \R^\dimension$, we have  
\[
|\indicator{Q_j} \convolvedwith g_\time(x) - |Q_j| g_\time(x-y_k)| 
< 
|Q_j| \epsilon 
\]
by the uniform continuity of $(g_\time)_{\time \in [\largetime]}$. 
This implies for each $x \in \R^\dimension$ 
\[
\big| {f \convolvedwith g_\time(x) - \sum_j b_j |Q_j| g_\time(x-y_k)} \big|
< 
\sum_j |b_j| |Q_j| \epsilon 
= 
\Lpnorm{1}{f} \epsilon 
.\]
This completes the proof. 
\end{proof}

\begin{proof}[Proof of Theorem~\ref{theorem:CdG_variations}]
Fix \(r,p \geq 1\). Assume that the variation operator \(\variation(T_\time : \time \in \N)\) is pointed weak-type \((p,p)\) with norm at most \(C\). 
Our task is to show that \(\variation(T_\time : \time \in \N)\) is weak-type \((p,p)\) with norm at most \(C\). 

We commence with several standard reductions which we outline. 
The first step is to reduce to the truncated variation operators $\variation(T_\time : {\time \in [\largetime]})$ for arbitrarily large but finite $\largetime$. 
Since \(\variation(T_\time : \time \in \N)\) is pointed weak-type \((p,p)\) with norm at most \(C\), so is \(\variation(T_\time : \time \in [\largetime])\) is pointed weak-type \((p,p)\) with norm at most \(C\). It suffices to show that \(\variation(T_\time : \time \in [\largetime])\) is weak-type \((p,p)\) with norm at most \(C\). 
% Our results will be independent of $\largetime$ so this is fine. 
%
The second step is to boost \eqref{eq:restricted_pointed_weak_type_inequality} to the same inequality with arbitrary positive coefficients $a_k > 0$: 
\begin{equation}\label{eq:pointed_weak_type_inequality}
\Big| \big\{ 
x: \variation(\sum_{k} a_k g_\time(x-x_k) : \time \in [\largetime]) > \lambda
\big\} \Big|
\leq C \big( {\sum_{k} a_k^p} \big) \lambda^{-p} 
.\end{equation}
This step follows a standard technique: 
First prove \eqref{eq:pointed_weak_type_inequality} for $a_k \in \Z$. Then extend to rational coefficients. Finish by taking limits to conclude it for real coefficients. 
The third step is to reduce to smooth $g_\time \in L^1$ using sublinearity of the variation operators as in the proof of Theorem~\ref{theorem:Moon_variations}.  
At this point, we may now assume that for all $\epsilon>0$, there exists a $\delta>0$ depending on $\epsilon$ such that \( |g_{\time_1}(x)-g_{\time_2}(y)| < \epsilon \) for all $1 \leq \time_1, \time_2 \leq \largetime$ and all $|x-y|<\delta$. 
% Let $\epsilon>0$ and choose $\delta$ so that this holds. 

Suppose that $f := \sum_{k=1}^K a_k \indicator{Q_k}$ is a simple function where the $Q_k$ are dyadic cubes. 
Applying Proposition~\ref{proposition:CdG_approximation}, we may assume that all the dyadic cubes $Q_k$ have the same sidelength $\delta \leq 1$ and that \eqref{CdG:pointwise_comparison} holds true. 
For each $1 \leq k \leq K$, let $x_k$ be a fixed point in the interior of $Q_k$ e.g., the center of the cube. 
Define the functions 
\[
h_\time(x) 
:= \sum_{k=1}^K a_k |Q_k| g_\time(x-x_k) 
= \sum_{k=1}^K a_k |Q_k| T_\time \delta_{x_k}(x)
\]
for \(\time \in [\largetime]\). 
Then 
\begin{align*}
f*g_\time(x)-h_\time(x) %\sum_{k=1}^K a_k |Q_k| g_\time(x-x_k) 
&= 
\int \sum_{k=1}^K a_k \indicator{Q_k}(y) g_\time(x-y)-\sum_{k=1}^K a_k |Q_k| g_\time(x-x_k) 
\\&= 
\sum_{k=1}^K a_k \int \indicator{Q_k}(y) g_\time(x-y) - \indicator{Q_k}(y) g_\time(x-x_k) \; dy 
\end{align*}
Taking absolute values and applying the triangle inequality, we obtain 
\begin{align*}
|f*g_\time(x)-h_\time(x)| 
&\leq 
\sum_{k=1}^K |a_k| \int_{Q_k} |g_\time(x-y) - g_\time(x-x_k)| \; dy 
\\&\leq 
\sum_{k=1}^K |a_k| |Q_k| \epsilon 
= 
\|f\|_1 \epsilon
.
\end{align*}
% Then 
% \begin{align*}
% \variation(f*g_\time(x) : \time \in [\largetime]) 
% & \leq 
% \variation(f*g_\time(x)-\sum_{k} a_k |Q_k| g_\time(x-x_k) : \time \in [\largetime]) 
% \\
% & \hspace{10mm}
% + \variation(\sum_{k} a_k |Q_k| g_\time(x-x_k) : \time \in [\largetime])
% .\end{align*}
Choosing $\epsilon = \epsilon'/(8\largetime^{2} \|f\|_1)$ and applying the inequality \eqref{bound:union} yields 
\begin{align*}
|\setof{\variation(f*g_\time(x) : \time \in [\largetime]) > \lambda + \epsilon'}| 
\leq 
\left| \setof{\variation(h_\time(x) : \time \in [\largetime]) > \lambda} \right|
.\end{align*}
Since \( h_\time(x) = \sum_{k=1}^K a_k |Q_k| T_\time \delta_{x_k}(x) \), applying the boosted pointed weak-type hypothesis \eqref{eq:pointed_weak_type_inequality} implies that 
\begin{align*}
\left| \setof{\variation(h_\time(x) : \time \in [\largetime]) > \lambda} \right|
\leq 
C \Big( \sum_{k=1}^K |a_k|^p |Q_k|^p \Big)^{1/p} \lambda^{-p}
% \leq 
% C \cdot \|f\|_p^p \lambda^{-p}
.\end{align*}
Upon letting \(\epsilon'\) tend to 0, it suffices to show that 
\(
\Big( \sum_{k=1}^K |a_k|^p |Q_k|^p \Big)^{1/p} 
\leq \|f\|_p 
.\)
This follows because \(\delta \leq 1\) and \(p \geq 1\) which implies \(\delta \geq \delta^p\) and 
\[
\|f\|_p 
= 
\Big( \sum_{k=1}^K |a_k|^p \delta^\dimension \Big)^{1/p} 
\geq 
\Big( \sum_{k=1}^K |a_k|^p \delta^{p\dimension} \Big)^{1/p} 
=
\Big( \sum_{k=1}^K |a_k|^p |Q_k|^p \Big)^{1/p} 
.\]

The final step is to extend from simple functions formed by the standard dyadic mesh on $\R^\dimension$ to general functions in $L^p(\R^\dimension)$ by adapting the argument at the end of the proof of Theorem~\ref{theorem:Moon_variations}. 
The modifications for jump inequalities are like those for Theorems~\ref{theorem:Moon_variations} and \ref{theorem:Moon_smoothing}. 
%Make similar reductions as above such as reducing to the truncated jump function $\jumpfunction_\jump(f*\mu_\time : {\time \in [\largetime]})$, and replace \eqref{} with 
%\begin{align*}
%\mathcal{N}_{\lambda}(f*g_\time(x) : \time \in [\largetime]) 
%& \leq 
%\mathcal{N}_{\lambda/2}(f*g_\time(x)-\sum_{k} a_k |Q_k| g_\time(x-x_k) : \time \in [\largetime]) 
%\\
%& \hspace{10mm}
%+ \mathcal{N}_{\lambda/2}(\sum_{k} a_k |Q_k| g_\time(x-x_k) : \time \in [\largetime])
%.\end{align*}
We leave the details to the reader. 
\end{proof}

% 
% How to make a short bibliography
% 
% -----
% \newpage
% \begin{thebibliography}{9}
% 
% \bibitem{CSdiscMultiplier} L. Carleson and P. Sjolin, 
% \emph{Oscillatory integrals and a multiplier problem for the disc}, 
% Studia Math. 44 (1972), 287-299.
% 
% \end{thebibliography}
%

% \bibliographystyle{amsplain}

% \providecommand{\bysame}{\leavevmode\hbox to3em{\hrulefill}\thinspace}
% \providecommand{\MR}{\relax\ifhmode\unskip\space\fi MR }
% % \MRhref is called by the amsart/book/proc definition of \MR.
% \providecommand{\MRhref}[2]{%
%   \href{http://www.ams.org/mathscinet-getitem?mr=#1}{#2}
% }
% \providecommand{\href}[2]{#2}

\end{document}